\documentclass[leqno]{amsart}
\usepackage{amsmath}
\usepackage{amssymb}
\usepackage{amsthm}
\usepackage{enumerate}
\usepackage[mathscr]{eucal}
\usepackage{enumitem}
\theoremstyle{plain}
\newtheorem{theorem}{Theorem}[section]
\newtheorem{lemma}[theorem]{Lemma}
\newtheorem{prop}[theorem]{Proposition}
\theoremstyle{definition}

\newtheorem{remark}[theorem]{Remark}

\newtheorem{cor}[theorem]{Corollary}
\theoremstyle{remark}
\usepackage{amsmath, amsthm, amscd, amsfonts, amssymb, graphicx, color}
\usepackage[bookmarksnumbered, colorlinks, plainpages]{hyperref}
\hypersetup{colorlinks=true,linkcolor=red, anchorcolor=green, citecolor=cyan, urlcolor=red, filecolor=magenta, pdftoolbar=true}

\begin{document}
\title [Numerical radius inequalities of operator matrices ]{Numerical radius inequalities of operator matrices from a new norm on $\mathcal{B}(\mathcal{H})$} 

\author[P. Bhunia, A. Bhanja, D. Sain and K. Paul ] {P. Bhunia, A. Bhanja, D. Sain and K. Paul  }

\address[Bhunia] { Department of Mathematics, Jadavpur University, Kolkata 700032, West Bengal, India}
\email{pintubhunia5206@gmail.com}

\address[Bhanja]{Department of Mathematics\\ Vivekananda College Thakurpukur\\ Kolkata\\ West Bengal\\India\\ }
\email{aniketbhanja219@gmail.com}

\address[Sain]{Department of Mathematics\\ Indian Institute of Science\\ Bengaluru 560012\\ Karnataka\\ INDIA}
\email{saindebmalya@gmail.com}

\address[Paul] {Department of Mathematics, Jadavpur University, Kolkata 700032, West Bengal, India}
\email{kalloldada@gmail.com}

\thanks{ Mr. Pintu Bhunia would like to thank UGC, Govt. of India for the financial support in the form of SRF. 
}



\maketitle
\begin{abstract}
	This paper is a continuation of a recent work on a new norm, christened the $ (\alpha, \beta)$-norm, on the space of bounded linear operators on a Hilbert space. We obtain some upper bounds for the said norm of $n\times n$ operator matrices. As an application of the present study, we estimate bounds for the numerical radius and the usual operator norm of $n\times n$ operator matrices, which generalize the existing ones. 
\end{abstract}

\section{\textbf{Introduction}}
\noindent The purpose of the present article is to study the bounds for the newly introduced \cite{SBBP} $(\alpha,\beta)$-norm of $n\times n$ operator matrices, from which we obtain bounds for the numerical radius of $n\times n$ operator matrices. Let us first introduce the following notations and terminologies to be used throughout the article.\\  

\noindent Let $\mathcal{H}_i, \mathcal{H}_j$ be two complex Hilbert spaces with usual inner product $\langle.,.\rangle$ and let $\mathcal{B}(\mathcal{H}_i,\mathcal{H}_j)$ denote the space of all bounded linear operators from $\mathcal{H}_i$ to $\mathcal{H}_j$. If $\mathcal{H}_i= \mathcal{H}_j=\mathcal{H}$ then we write $\mathcal{B}(\mathcal{H},\mathcal{H})=\mathcal{B}({\mathcal{H}}).$ For $T\in \mathcal{B}({\mathcal{H}})$, we write $Re(T)$ and $Im(T)$ for the real part of $T$ and the imaginary part of $T$, respectively, i.e., $Re(T)=\frac{T+T^*}{2}$ and $Im(T)=\frac{T-T^*}{2i}$. Let $T^*$ denote the adjoint of $T$ and let $|T|$ be the positive operator $(T^*T)^{\frac{1}{2}}$. Let $\sigma (T)$ denote the spectrum of $T$. The spectral radius of $T$, denoted by $r(T)$, is defined by $r(T)=\sup \{|\lambda|: \lambda \in \sigma (T)\}.$ 
The numerical range of $T$,  denoted by $W(T)$,  is  defined as $W(T)=\{\langle Tx,x \rangle:x\in \mathcal{H}, \|x\|=1\}.$ The usual operator norm and the numerical radius of $T$, denoted by $\|T\|$ and $ w(T)$, respectively,  are defined as $\|T\|=\sup \{\|Tx\| : x\in \mathcal{H}, \|x\|=1\}$ and $w(T)= \sup \{|c|: c\in W(T) \}.$ Let $M_T$ denote the usual operator norm attainment set of $T$, i.e., $M_T=\{x\in \mathcal{H}: \|Tx\|=\|T\|, \|x\|=1\}$. 
It is well-known that the numerical radius defines a norm on $\mathcal{B}({\mathcal{H}})$ and is equivalent to the usual operator norm, satisfying that for $T\in \mathcal{B}({\mathcal{H}}),$
\[\frac{1}{2}\|T\| \leq w(T)\leq \|T\|.\] 
The study of the numerical range of an operator and the associated numerical radius inequalities are an important area of research in operator theory and it has attracted many mathematicians \cite{OK,PSK,BBP4,BP,BBP,HD} over the years. With an aim to develop better upper and lower bounds for the numerical radius, a new norm named as the $(\alpha, \beta)$-norm, was introduced on $ \mathcal{B}({\mathcal{H}}) $ in \cite{SBBP}. For $T\in \mathcal{B}({\mathcal{H}}),$ the $(\alpha,\beta)$-norm of $T$, denoted by $\|T\|_{\alpha,\beta}$, is defined as: 
$$\|T\|_{\alpha,\beta}=\sup \left \{\sqrt{  \alpha |\langle Tx ,x \rangle |^2+\beta \|Tx\|^2 }: x\in \mathcal{H},\|x\|=1 \right \},$$
where $\alpha, \beta$ are real positive constants with $(\alpha ,\beta) \neq (0,0).$ We note that if $\alpha=1, \beta=0$ then $\|T\|_{\alpha,\beta}=w(T),$ and if $\alpha=0, \beta=1$ then $\|T\|_{\alpha,\beta}=\|T\|$. Also, if we consider $\alpha=\beta=1,$ then we have the modified Davis-Wielandt radius of $T$, that is,  $\|T\|_{\alpha,\beta}=dw^*(T)$, (see \cite{BSP}). In this article, we consider $ \alpha + \beta =1 $, i.e., $ \beta = 1 - \alpha$ and explore the $\alpha$-norm of $n \times n $ operator matrices, where the $\alpha$-norm of $T$ is defined as:
$$\|T\|_{\alpha}=\sup \left \{\sqrt{  \alpha |\langle Tx ,x \rangle |^2+(1-\alpha) \|Tx\|^2 }: x\in \mathcal{H},\|x\|=1 \right \}.$$

\noindent We compute the exact value of the $\alpha$-norm of $2\times 2$ operator matrices in $\mathcal{B}(\mathcal{H} \oplus \mathcal{H})$ of the form $\left(\begin{array}{cc}
	0 & X\\
	0 & 0
	\end{array}\right)$, where $X \in \mathcal{B}(\mathcal{H})$. We obtain some upper bounds for the $\alpha$-norm of $n \times n$ operator matrices, which generalize the existing numerical radius inequalities and the usual operator norm inequalities of $n \times n$ operator matrices.   As an application our results, we estimate new upper bounds for the numerical radius and the usual operator norm of $n \times n$ operator matrices.

\section{\textbf{Main results}}

\noindent We begin this section with the following proposition, the proof of which follows from the weakly unitarily invariant property of the $\alpha$-norm, i.e., for $T \in \mathcal{B}(\mathcal{H})$, $\|U^*TU\|_{\alpha}=\|T\|_{\alpha}$ for every unitary operator $U \in \mathcal{B}(\mathcal{H})$ (see \cite[Prop. 2.6]{SBBP}). 

\begin{prop}\label{lem1b}
	Let $A,B \in \mathcal{B}(\mathcal{H})$. Then the following results hold:
\begin{enumerate}[label=$(\alph*)$]
\item $\left \|\left(\begin{array}{cc}
	0 & A\\
	e^{{\rm i} \theta}B & 0
	\end{array}\right) \right \|_{\alpha}=\left \| \left(\begin{array}{cc}
	0 & A\\
	B & 0
	\end{array}\right) \right \|_{\alpha} $, for every $\theta \in \mathbb{R}.$
	\item $\left \|\left(\begin{array}{cc}
	0 & A\\
	B & 0
	\end{array}\right) \right \|_{\alpha} =\left \|\left(\begin{array}{cc}
	0 & B\\
	A & 0
	\end{array}\right) \right \|_{\alpha}.$ 
	\item $\left \|\left(\begin{array}{cc}
	A & 0\\
	0 & B
	\end{array}\right) \right \|_{\alpha}=\left \|\left(\begin{array}{cc}
	B & 0\\
	0 & A
	\end{array}\right) \right \|_{\alpha}.$ 
	\item $\left \|\left(\begin{array}{cc}
	A & B\\
	B & A
	\end{array}\right) \right \|_{\alpha}=\left \|\left(\begin{array}{cc}
	A-B & 0\\
	0 & A+B
	\end{array}\right) \right \|_{\alpha}.$
	\end{enumerate}
	\end{prop}
	
\noindent Next, we estimate upper and lower bounds for the $\alpha$-norm of $2\times 2$ operator matrices in $\mathcal{B}(\mathcal{H} \oplus \mathcal{H})$ of the form $\left(\begin{array}{cc}
	X & 0\\
	0 & Y
	\end{array}\right)$, where $X,Y \in \mathcal{B}(\mathcal{H})$. Let us first note the following inequality for $X \in \mathcal{B}(\mathcal{H}), $ $$\alpha|\langle Xx,x \rangle|^2+(1-\alpha)\|Xx\|^2\leq \|X\|^2_{\alpha}\|x\|^2\,\,\text{ for all }\,x\in \mathcal{H} \,\,\text{with}\,\,\|x\|\leq 1.$$

\begin{theorem}\label{est1}
Let $X,Y \in \mathcal{B}(\mathcal{H})$. Then the following inequalities hold:
\begin{eqnarray*}
(i)~~\max \left \{\|X\|_{\alpha}, \|Y\|_{\alpha} \right \} &\leq & \left \|\left(\begin{array}{cc}
	X & 0\\
	0 & Y
	\end{array}\right)\right \|_{\alpha} \\
	&\leq&  \max \left \{\sqrt{\|X\|_{\alpha}^2+\alpha w^2(X)}, \sqrt{\|Y\|_{\alpha}^2+\alpha w^2(Y)} \right \}\\
	&\leq & \sqrt{2}\max \left \{\|X\|_{\alpha}, \|Y\|_{\alpha} \right \}.
	\end{eqnarray*}
	$(ii)$ $\left \|\left(\begin{array}{cc}
	X & 0\\
	0 & Y
	\end{array}\right)\right \|_{\alpha} \leq  \sqrt{ \max \left \{ \|X\|_{\alpha}^2, \|Y\|_{\alpha}^2   \right \} +\alpha w(X)w(Y) }.$\\
	$(iii)$ $\left \|\left(\begin{array}{cc}
	X & 0\\
	0 & Y
	\end{array}\right)\right \|_{\alpha} \leq  \|X\|_{\alpha}+ \|Y\|_{\alpha}.$
\end{theorem}

\begin{proof}
$(i).$ Let $T=\left(\begin{array}{cc}
	X & 0\\
	0 & Y
	\end{array}\right)$. Let $x\in \mathcal{H} $ with $\|x\|=1$ and let $\tilde{x}=\left(\begin{array}{cc}
	x \\
	0 
	\end{array}\right) \in \mathcal{H}\oplus \mathcal{H}$. Clearly,  $\|\tilde{x}\|=1$.  Therefore, we have,
$$\sqrt{\alpha|\langle Xx,x \rangle|^2+(1-\alpha)\|Xx\|^2}=\sqrt{\alpha|\langle T\tilde{x},\tilde{x} \rangle|^2+(1-\alpha)\|T\tilde{x}\|^2}  \leq \|T\|_{\alpha}$$	
Taking supremum over all unit vector in $\mathcal{H}$, we get,  $$\|X\|_{\alpha} \leq \|T\|_{\alpha}.$$ Similarly, it can be proved that $$\|Y\|_{\alpha} \leq \|T\|_{\alpha}.$$ Combining the above two inequalities, we get the first inequality in $(i)$. Let us now prove the second inequality in $(i)$.
Let  $z=\left(\begin{array}{cc}
	x \\
	y 
	\end{array}\right) \in \mathcal{H}\oplus \mathcal{H} $ with $\|z\|=1$, i.e., $\|x\|^2+\|y\|^2=1$. Then we have,
	\begin{eqnarray*}
	\alpha|\langle Tz,z \rangle|^2+(1-\alpha)\|Tz\|^2 &=& \alpha \left|\langle Xx,x \rangle+ \langle Yy,y \rangle  \right|^2+(1-\alpha) (\|Xx\|^2+\|Yy\|^2 )\\
	&\leq& \alpha \left ( |\langle Xx,x \rangle|+ |\langle Yy,y \rangle| \right)^2+ (1-\alpha) \left ( \|Xx\|^2+\|Yy\|^2  \right )\\
	&\leq &  \alpha |\langle Xx,x \rangle|^2+(1-\alpha)\|Xx\|^2 +  \alpha |\langle Yy,y \rangle|^2+(1-\alpha)\|Yy\|^2   \\
	&& +\alpha \left ( |\langle Xx,x \rangle|^2+ |\langle Yy,y \rangle|^2 \right )\\
	&\leq& \|X\|_{\alpha}^2 \|x\|^2+\|Y\|_{\alpha}^2 \|y\|^2\\
	&& +\alpha \left ( w^2(X)\|x\|^2+ w^2(Y)\|y\|^2 \right ),~~ \mbox{since,} ~~ \|x\|\leq 1,\|y\|\leq 1\\
	&=& \left (\|X\|_{\alpha}^2+\alpha w^2(X)\right ) \|x\|^2+\left (\|Y\|_{\alpha}^2+\alpha w^2(Y)\right ) \|y\|^2\\
	&\leq& \max \left \{ \|X\|_{\alpha}^2+\alpha w^2(X), \|Y\|_{\alpha}^2+\alpha w^2(Y)\right \}.
	\end{eqnarray*}
Therefore, taking supremum over all unit vectors in $\mathcal{H}\oplus \mathcal{H}$, we get the second inequality in $(i)$. The remaining inequality in $(i)$ follows from the inequalities $\alpha w^2(X)\leq \|X\|_{\alpha}^2$ and $\alpha w^2(Y)\leq \|Y\|_{\alpha}^2$. This completes the proof of $(i)$.\\

$(ii).$  From $\alpha|\langle Tz,z \rangle|^2+(1-\alpha)\|Tz\|^2 \leq \alpha \left ( |\langle Xx,x \rangle|+ |\langle Yy,y \rangle| \right)^2+ (1-\alpha)\left ( \|Xx\|^2+\|Yy\|^2  \right )$, we get 
\begin{eqnarray*}
\alpha|\langle Tz,z \rangle|^2+(1-\alpha)\|Tz\|^2 &\leq& \alpha |\langle Xx,x \rangle|^2+(1-\alpha)\|Xx\|^2 +  \alpha |\langle Yy,y \rangle|^2+(1-\alpha)\|Yy\|^2   \\
	&& +2 \alpha |\langle Xx,x \rangle|~~ |\langle Yy,y \rangle| \\
	&\leq& \alpha |\langle Xx,x \rangle|^2+(1-\alpha)\|Xx\|^2 +  \alpha |\langle Yy,y \rangle|^2+(1-\alpha)\|Yy\|^2   \\
	&& + 2\alpha w(X)w(Y)\|x\|^2\|y\|^2 \\
	&\leq& \|X\|_{\alpha}^2 \|x\|^2+\|Y\|_{\alpha}^2 \|y\|^2\\
	&& +2\alpha w(X)w(Y)\|x\|\|y\|,~~ \mbox{since,} ~~ \|x\|\leq 1,\|y\|\leq 1\\ 
	&\leq& \max \left \{ \|X\|_{\alpha}^2, \|Y\|_{\alpha}^2   \right \} +\alpha w(X)w(Y).
\end{eqnarray*}
Taking supremum over all unit vectors in $\mathcal{H}\oplus \mathcal{H}$, we get the inequality in $(ii)$.\\

$(iii).$ The inequality in $(iii)$ follows from the triangle inequality of the $\alpha $-norm, and by using the inequality in $(ii)$.
\end{proof}

In the following theorem, we obtain the exact value of the $\alpha$-norm of $2\times 2$ operator matrices in $\mathcal{B}(\mathcal{H} \oplus \mathcal{H})$ of the form $\left(\begin{array}{cc}
	0 & X\\
	0 & 0
	\end{array}\right)$, where $X \in \mathcal{B}(\mathcal{H})$. 

\begin{theorem}\label{equal1}
Let $X \in \mathcal{B}(\mathcal{H})$. Then 
	\[ \left \|\left(\begin{array}{cc}
	0 & X\\
	0 & 0
	\end{array}\right)\right \|_{\alpha}=\begin{cases}
		\frac{1}{2\sqrt{\alpha}}~\|X\| &\text{if}~~\,\alpha >\frac{1}{2} \\
		\sqrt{1-\alpha}~\|X\| &\text{if}~~\,\alpha \leq \frac{1}{2}.
		\end{cases} \]	
\end{theorem}

\begin{proof}
Let $T =\left(\begin{array}{cc}
	0 & X\\
	0 & 0
	\end{array}\right)$. Let  $z=\left(\begin{array}{cc}
	x \\
	y 
	\end{array}\right) \in \mathcal{H}\oplus \mathcal{H} $ with $\|z\|=1$, i.e., $\|x\|^2+\|y\|^2=1$.
Then $ \langle Tz,z \rangle = \langle Xy,x \rangle ~~\mbox{and} ~~\|Tz\|= \|Xy\|.$ Now we have,
\begin{eqnarray*}
	\|T\|_{\alpha}^2=\sup_{\|z\|=1}(\alpha|\langle Tz,z \rangle|^2+(1-\alpha)\|Tz\|^2) &=& \sup_{\|x\|^2+\|y\|^2=1}(\alpha|\langle Xy,x \rangle|^2 +(1-\alpha)\|Xy\|^2) \\
	&\leq & \sup_{\|x\|^2+\|y\|^2=1} (\alpha \|X\|^2\|y\|^2\|x\|^2+ (1-\alpha) \|X\|^2\|y\|^2)\\
	&=& \sup_{\theta \in [0,\frac{\pi}{2}]} \|X\|^2 \sin^2 \theta (\alpha \cos^2\theta +(1-\alpha)). 
	\end{eqnarray*}
First we consider the case $\alpha>\frac{1}{2}.$ Then
$$\sup_{\theta \in [0,\frac{\pi}{2}]} \|X\|^2 \sin^2 \theta (\alpha \cos^2\theta +(1-\alpha))= \frac{1}{4\alpha} \|X\|^2.$$
Therefore, $\|T\|_{\alpha}^2 \leq \frac{1}{4\alpha} \|X\|^2.$ We claim that there exists a sequence $\{z_n\}$ in $\mathcal{H} \oplus \mathcal{H}$ with $\|z_n\|=1$ such that 
$$\lim_{n \to \infty} (\alpha |\langle Tz_{n},z_{n} \rangle|^2 + (1-\alpha) \|Tz_{n}\|^2) =\frac{1}{4\alpha} \|X\|^2.$$ 
\noindent Clearly, there exists a sequence $\{y_n \}$ in $\mathcal{H}$ with $\|y_n\|=1$ such that $\lim_{n \to \infty} \|Xy_n\| = \|X\|.$
Let $z_{n}=\frac{1}{\sqrt{\|Xy_n\|^2+k^2}}\left(\begin{array}{cc}
	Xy_n \\
	ky_n 
	\end{array}\right)$, where $k = \sqrt{ \frac{1}{2\alpha -1}}\|X\|$. Then
	$$\lim_{n \to \infty} \alpha |\langle Tz_{n},z_{n} \rangle|^2 + (1-\alpha) \|Tz_{n}\|^2 =\frac{1}{4\alpha} \|X\|^2.$$ 
Therefore, $\|T\|_{\alpha} = \frac{1}{2\sqrt{\alpha}} \|X\|$ if $\alpha>\frac{1}{2}.$\\
Next we consider the case $\alpha \leq \frac{1}{2}$. Then 
 $$\sup_{\theta \in [0,\frac{\pi}{2}]} \|X\|^2 \sin^2 \theta (\alpha \cos^2\theta +(1-\alpha))= (1-\alpha) \|X\|^2$$
Therefore, $\|T\|_{\alpha}^2 \leq (1-\alpha)\|X\|^2$. Proceeding as before, we can show that there exists a sequence $ \{ z_n \}, \|z_n\|=1 $ such that  $\lim_{n \to \infty}(\alpha|\langle Tz_n,z_n \rangle|^2+(1-\alpha) \|Tz_n\|^2)={(1-\alpha)} \|X\|^2.$ Therefore, $\|T\|_{\alpha}= \sqrt{(1-\alpha)} \|X\|$ if $ \alpha \leq \frac{1}{2}$. 
\end{proof}



Our next goal is to obtain upper bounds for the $\alpha$-norm of $n\times n$ operator matrices in $\mathcal{B}(\oplus_{i=1}^n \mathcal{H}_i)$. We require the following lemmas for our purpose.

\begin{lemma}$($\cite[p. 44]{HJ}$)$\label{lem}
Let $T=(t_{ij})\in M_n(\mathbb{C})$ with $t_{ij}\geq 0$ for all $i,j$. Then $$w(T)=r\left (Re(T)\right)=\|Re(T)\|.$$
\end{lemma}

\begin{lemma}$($\cite{K88}$)$\label{lem9}
Let $T \in \mathcal{B}(\mathcal{H})$ be self-adjoint and let $x\in \mathcal{H}$. Then
\[|\langle Tx,x \rangle | \leq \langle |T|x,x \rangle.\]
\end{lemma}

\begin{lemma}$($\cite{K88}$)$\label{lem5}
 Let $T \in \mathcal{B}({\mathcal{H}})$ with $T \geq 0$  and let $x \in \mathcal{H}$ with $\|x\|=1$. Then
\[\langle Tx,x \rangle ^p \leq \langle T^px,x \rangle,~~ \forall ~p \geq 1.\]  
\end{lemma}

\begin{lemma}$($\cite[Th. 2.1]{SBBP}$)$\label{norm}
 Let $T \in \mathcal{B}({\mathcal{H}})$. Then the following inequalities hold:
\[ w(T) \leq  \|T\|_{\alpha} \leq \sqrt{4-3\alpha } ~w(T), \]
\[ \max \left \{\frac{1}{2},\sqrt{(1-\alpha)}\right \}\|T\| \leq \|T\|_{\alpha} \leq ~\|T\|. \]
\end{lemma}

Now we are in a position to prove the following inequality.

\begin{theorem}\label{est6}
Let $\mathcal{H}_1,\mathcal{H}_2,\ldots,\mathcal{H}_n$ be Hilbert spaces. Let $T=(T_{ij})$ be an $n\times n$ operator matrix, where $T_{ij}\in \mathcal{B}(\mathcal{H}_j,\mathcal{H}_i)$. Then 
$$ \|T\|_{\alpha} \leq \sqrt{ \left \|\alpha |R|^2+(1-\alpha) |S|^{2}\right \|} ,$$
where $R=(r_{ij})_{n \times n}$, $ r_{ij}=\begin{cases}
		w(T_{ij}) &\text{if }~i=j \\
		\frac{1}{2}(\|T_{ij}\|+\|T_{ji}\|) &\text{if} ~i\neq j  
		\end{cases}	$\\
		 and $S=(s_{ij})_{n \times n}$, $s_{ij} =\|T_{ij}\|$.
\end{theorem}

\begin{proof}
Let $x=(x_1,x_2,...,x_n) \in \oplus_{i=1}^n \mathcal{H}_i$ with $\|x\|=1$ and let $\tilde{x}=(\|x_1\|,\|x_2\|,...,\|x_n\|)$. Clearly, $\tilde{x}$ is a unit vector in $\mathbb{C}^n$. Now,
\begin{eqnarray*}
|\langle Tx,x \rangle| &=& \left |\sum_{i,j=1}^{n} \langle T_{ij}x_j,x_i \rangle \right |\\
&\leq& \sum_{i,j=1}^{n} |\langle T_{ij} x_j,x_i \rangle | \\
&\leq& \sum_{i=1}^{n} |\langle T_{ii} x_i,x_i \rangle | +\sum_{i,j=1 ;i\neq j}^{n}|\langle T_{ij} x_j,x_i \rangle |\\
&\leq& \sum_{i=1}^{n} w(T_{ii}) \|x_i\|^2 +\sum_{i,j=1 ;i\neq j}^{n} \|T_{ij}\| \|x_j\|\|x_i\| \\
&=& \sum_{i,j=1}^{n} \tilde{t_{ij}} \|x_j\|\|x_i\| \\
&=& \langle \tilde{T} \tilde{x} , \tilde{x} \rangle\\
&=& \langle Re(\tilde{T}) \tilde{x},\tilde{x} \rangle  + i\langle  Im(\tilde{T})\tilde{x},\tilde{x}  \rangle,
\end{eqnarray*}
\noindent where $\tilde{T}=(\tilde{t_{ij}})$, $\tilde{t_{ij}}=\begin{cases}
		w(T_{ij}) &\text{if} ~i=j \\
		\|T_{ij}\| &\text{if} ~i\neq j.  \\
		\end{cases}$ \\
		 Clearly, $\langle  Im(\tilde{T})\tilde{x},\tilde{x}  \rangle=0$. So by using Lemma \ref{lem9} and Lemma \ref{lem5}, we get
\begin{eqnarray*}
|\langle Tx,x \rangle| &\leq& \langle Re(\tilde{T}) \tilde{x},\tilde{x}  \rangle\leq  \langle |Re(\tilde{T}) |\tilde{x},\tilde{x}  \rangle	\\
\Rightarrow 	|\langle Tx,x \rangle|^2 &\leq&  \langle |Re(\tilde{T}) |\tilde{x},\tilde{x}  \rangle ^2 \leq  \langle |Re(\tilde{T}) |^2\tilde{x},\tilde{x}  \rangle 	=\langle |R|^2\tilde{x},\tilde{x}  \rangle.
\end{eqnarray*}

Also,  
\begin{eqnarray*}
\|Tx\|^2 &=& |\langle Tx,Tx \rangle |\\
&=& \left |\sum_{i,j,k=1}^{n} \langle T_{kj} x_j,T_{ki} x_i  \rangle \right |\\
&\leq& \sum_{i,j,k=1}^{n} |\langle T_{kj}x_j,T_{ki} x_i  \rangle |\\
&\leq& \sum_{i,j,k=1}^{n} |\langle T_{ki}^*T_{kj}x_j, x_i  \rangle |\\
&\leq& \sum_{i,j,k=1}^{n} \|T_{ki}\|\|T_{kj}\|\|x_j\|\|x_i\|\\
&=& \langle |S|^{2} \tilde{x},\tilde{x}  \rangle .
\end{eqnarray*}
Therefore, 
\begin{eqnarray*}
  \alpha |\langle Tx ,x \rangle |^2+(1-\alpha) \|Tx\|^2  &\leq& \alpha \langle |R|^2 \tilde{x},\tilde{x}  \rangle + (1-\alpha) \langle |S|^{2} \tilde{x},\tilde{x}  \rangle \\
&=&  \langle \left (\alpha |R|^2+(1-\alpha) |S|^{2}   \right)  \tilde{x},\tilde{x} \rangle \\
&\leq & \left \| \alpha |R|^2+(1-\alpha) |S|^{2}  \right \|.
\end{eqnarray*}
Taking supremum over all unit vectors in $\oplus_{i=1}^n \mathcal{H}_i$, we get the desired inequality.
\end{proof}

As a consequence of  Theorem \ref{est6}, the following numerical radius inequality and the usual operator norm inequality can be proved quite easily. 

\begin{cor}\label{est6_c1}
Let $\mathcal{H}_1,\mathcal{H}_2,\ldots,\mathcal{H}_n$ be Hilbert spaces. Let $T=(T_{ij})$ be an $n\times n$ operator matrix, where $T_{ij}\in \mathcal{B}(\mathcal{H}_j,\mathcal{H}_i)$. Then
\begin{eqnarray*}
(i)~~w(T) &\leq& \min_{0\leq \alpha\leq 1} \sqrt{\left \|\alpha |R|^2+(1-\alpha) |S|^{2}\right \|} \leq w(\tilde{T}) \\
(ii)~~\|T\| &\leq& \min_{0\leq \alpha\leq 1} \frac{1}{\max \left \{\frac{1}{2},\sqrt{1-\alpha}\right \}} \sqrt{\left \|\alpha |R|^2+(1-\alpha)|S|^{2}\right \|} \leq \|S\|,
\end{eqnarray*}
 where $\tilde{T}=(\tilde{t_{ij}})_{n \times n}$, $ \tilde{t_{ij}}=\begin{cases}
		w(T_{ij}) &\text{if }~i=j \\
		\|T_{ij}\| &\text{if }~i\neq j  
		\end{cases}	$ \\
		and $R,S$ are same as described in Theorem \ref{est6}.
\end{cor}



We would like to note that the inequalities in \cite[Th. 1]{OK} and \cite[Th. 1.1]{HD} follow from (i) and (ii) of Corollary \ref{est6_c1}, respectively.\\

In our next result, we obtain an upper bound for the $\alpha$-norm of $n \times n$ operator matrices in terms of non-negative continuous functions on $[0,\infty)$. First we need the following lemma.

\begin{lemma}	$($\cite[Th. 5]{K88}$)$  \label{lem7}
Let $T\in \mathcal{B}(\mathcal{H})$ and let $f$ and $g$ be two non-negative continuous functions on $[0,\infty)$ such that $f(t)g(t)=t,~~\forall~~ t\in [0,\infty).$ Then
$$|\langle Tx,y \rangle|\leq \|f(|T|)x\| \|g(|T^*|)y\|, ~~\forall~~x,y\in \mathcal{H}.$$
	\end{lemma}
	
\begin{theorem}\label{est7}
Let $T=(T_{ij})$ be an $n\times n$ operator matrix, where $T_{ij}\in \mathcal{B}(\mathcal{H})$. Let $f$ and $g$ be two non-negative continuous functions on $[0,\infty)$ such that $f(t)g(t)=t,$ $\forall ~t \geq 0$. Then
\begin{eqnarray*}
\|T\|_{\alpha} &\leq& \sqrt{\left \|\alpha |R|^2+(1-\alpha) |S|^{2}\right \|}.
\end{eqnarray*}
where $R=(r_{ij})_{n \times n}$, $ r_{ij}=\frac{1}{2}\left (\|f^2(|T_{ij}|)\|^{\frac{1}{2}} \|g^2(|T_{ij}^*|)\|^{\frac{1}{2}}+\|f^2(|T_{ji}|)\|^{\frac{1}{2}} \|g^2(|T_{ji}^*|)\|^{\frac{1}{2}} \right )$ and $S=(s_{ij})_{n \times n}$, $s_{ij} =\|T_{ij}\|$.
\end{theorem}

\begin{proof}
Let $x=(x_1,x_2,...,x_n) \in \oplus_{i=1}^n \mathcal{H}$ with $\|x\|=1$ and let $\tilde{x}=(\|x_1\|,\|x_2\|,...,\|x_n\|)$. Clearly, $\tilde{x}$ is a unit vector in $\mathbb{C}^n$. Using Lemma \ref{lem7}, we get that
\begin{eqnarray*}
|\langle Tx,x \rangle| &=& \left |\sum_{i,j=1}^{n} \langle T_{ij}x_j,x_i \rangle \right |\\
&\leq& \sum_{i,j=1}^{n} |\langle T_{ij} x_j,x_i \rangle | \\
&\leq& \sum_{i,j=1}^{n} \|f(|T_{ij}|)x_j\|\|g(|T_{ij}^*|)x_i\|\\
&=& \sum_{i,j=1}^{n} \langle f^2(|T_{ij}|)x_j,x_j \rangle ^{\frac{1}{2}}  \langle g^2(|T_{ij}^*|)x_i,x_i  \rangle ^{\frac{1}{2}}\\
&\leq& \sum_{i,j=1}^{n} \|f^2(|T_{ij}|)\|^{\frac{1}{2}} \|g^2(|T_{ij}^*|)\|^{\frac{1}{2}} \|x_i\|\|x_j\| \\
&=& \sum_{i,j=1}^{n} \tilde{t_{ij}} \|x_j\|\|x_i\| \\
&=& \langle \tilde{T} \tilde{x} , \tilde{x} \rangle\\
&=& \langle Re(\tilde{T}) \tilde{x},\tilde{x} \rangle  +i\langle  Im(\tilde{T})\tilde{x},\tilde{x}  \rangle,
\end{eqnarray*}
\noindent where $\tilde{T}=(\tilde{t_{ij}})$, $\tilde{t_{ij}}=\|f^2(|T_{ij}|)\|^{\frac{1}{2}} \|g^2(|T_{ij}^*|)\|^{\frac{1}{2}}.$\\
 Proceeding similarly as in the proof of Theorem \ref{est6}, we get 
\[ |\langle Tx,x \rangle|^2 \leq  \langle |R|^2\tilde{x},\tilde{x}  \rangle  ~~\mbox{and}~~\|Tx\|^2  \leq \langle |S|^{2} \tilde{x},\tilde{x}  \rangle.	\]
Therefore,
\begin{eqnarray*}
 \alpha |\langle Tx ,x \rangle |^2+(1-\alpha) \|Tx\|^2  &\leq & \left \| \alpha |R|^2+(1-\alpha) |S|^{2}  \right \|.
\end{eqnarray*}
Taking supremum over all unit vectors in $\oplus_{i=1}^n \mathcal{H}$, we get the desired inequality.
\end{proof}
 
The following numerical radius inequality is an easy consequence of Theorem \ref{est7} .

\begin{cor}\label{est7_c1}
Let $T=(T_{ij})$ be an $n\times n$ operator matrix, where $T_{ij}\in \mathcal{B}(\mathcal{H})$. Let $f$ and $g$ be non-negative continuous functions on $[0,\infty)$ such that $f(t)g(t)=t,$ $\forall ~t \geq 0$. Then
\begin{eqnarray*}
w(T) &\leq& \min_{0\leq \alpha \leq 1}\sqrt{ \left \|\alpha |R|^2+(1-\alpha)|S|^{2}\right \|}\leq w(\tilde{T}),
\end{eqnarray*}
where $R,S$ are same as described in Theorem \ref{est7} and  $\tilde{T} =(\tilde{t_{ij}})_{n\times n}$, $\tilde{t_{ij}}=\|f^2(|T_{ij}|)\|^{\frac{1}{2}} \|g^2(|T_{ij}^*|)\|^{\frac{1}{2}}.$ 
\end{cor}



We would like to note that the inequality in \cite[Th. 3.1]{BP} follows from Corollary \ref{est7_c1}.\\

In our next theorem, we obtain an upper bound for the $\alpha $-norm of $n\times n$ operator matrices.

\begin{theorem}\label{est8}
Let $T=(T_{ij})$ be an $n\times n$ operator matrix, where $T_{ij}\in \mathcal{B}(\mathcal{H})$. Let $f,g$ be two non-negative continuous functions on $[0,\infty)$ such that $f(t)g(t)=t,$ $\forall ~ t \geq 0 $. If $p \geq 1$, then
\[ \|T\|_{\alpha}^p \leq \sqrt{\left \|\alpha |R|^{2p}+(1-\alpha) |S|^{2p} \right \|},\]
where $ r_{ij}=\begin{cases}
		\frac{1}{2}\left \| f^2(|T_{ii}|)  + g^2(|T_{ii}^*|) \right \| &\text{if} ~i=j \\
		\frac{1}{2} \left (\|f^2(|T_{ij}|)\|^{\frac{1}{2}} \|g^2(|T_{ij}^*|)\|^{\frac{1}{2}}+\|f^2(|T_{ji}|)\|^{\frac{1}{2}} \|g^2(|T_{ji}^*|)\|^{\frac{1}{2}} \right ) &\text{if} ~i\neq j,  \\
		\end{cases}$\\ $R=(r_{ij})_{n \times n}$ and  $S=(s_{ij})_{n \times n}$, $s_{ij} =\|T_{ij}\|$.
\end{theorem}

\begin{proof}
Let $x=(x_1,x_2,...,x_n) \in \oplus_{i=1}^n \mathcal{H}$ with $\|x\|=1$ and let $\tilde{x}=(\|x_1\|,\|x_2\|,...,\|x_n\|)$. Clearly, $\tilde{x}$ is a unit vector in $\mathbb{C}^n$. Using Lemma \ref{lem7}, we get that
\begin{eqnarray*}
|\langle Tx, x\rangle | &=& \left |\sum_{i,j=1}^{n} \langle T_{ij}x_j,x_i \rangle \right |\\
&\leq& \sum_{i,j=1}^{n} |\langle T_{ij} x_j,x_i \rangle | \\
&\leq& \sum_{i,j=1}^{n} \|f(|T_{ij}|)x_j\|\|g(|T_{ij}^*|)x_i\|\\
&=& \sum_{i,j=1}^{n} \langle f^2(|T_{ij}|)x_j,x_j \rangle ^{\frac{1}{2}}  \langle g^2(|T_{ij}^*|)x_i,x_i  \rangle ^{\frac{1}{2}}\\
&\leq&  \sum_{i=1}^{n} \frac{1}{2} \left (\langle f^2(|T_{ii}|)x_i,x_i \rangle   + \langle g^2(|T_{ii}^*|)x_i,x_i \rangle  \right ) \\
&&+ \sum_{i,j=1,i\neq j}^{n} \langle f^2(|T_{ij}|)x_j,x_j \rangle ^{\frac{1}{2}}  \langle g^2(|T_{ij}^*|)x_i,x_i  \rangle ^{\frac{1}{2}}\\
&\leq&  \sum_{i=1}^{n} \frac{1}{2} \langle \left ( f^2(|T_{ii}|) +  g^2(|T_{ii}^*|)\right )x_i,x_i \rangle   \\
&&  +\sum_{i,j=1,i\neq j}^{n} \langle f^2(|T_{ij}|)x_j,x_j \rangle ^{\frac{1}{2}}  \langle g^2(|T_{ij}^*|)x_i,x_i  \rangle ^{\frac{1}{2}}\\
&\leq&  \sum_{i=1}^{n}  \frac{1}{2}\left \| f^2(|T_{ii}|)  + g^2(|T_{ii}^*|) \right \| \|x_i\|^2 \\
&&+\sum_{i,j=1,i\neq j}^{n} \|f^2(|T_{ij}|)\|^{\frac{1}{2}} \|g^2(|T_{ij}^*|)\|^{\frac{1}{2}} \|x_i\|\|x_j\| 
\end{eqnarray*}

\begin{eqnarray*}
&=& \sum_{i,j=1}^{n} \tilde{t_{ij}} \|x_j\|\|x_i\| \\
&=& \langle \tilde{T} \tilde{x} , \tilde{x} \rangle\\
&=& \langle Re(\tilde{T}) \tilde{x},\tilde{x} \rangle  +i\langle  Im(\tilde{T})\tilde{x},\tilde{x}  \rangle,
\end{eqnarray*}
\noindent where $\tilde{T}=(\tilde{t_{ij}})_{n \times n}$, $\tilde{t_{ij}}=\begin{cases}
		\frac{1}{2}\left \| f^2(|T_{ii}|)  + g^2(|T_{ii}^*|) \right \| &\text{if} ~i=j \\
		\|f^2(|T_{ij}|)\|^{\frac{1}{2}} \|g^2(|T_{ij}^*|)\|^{\frac{1}{2}} &\text{if} ~i\neq j.  \\
		\end{cases}	$ \\
		Clearly, $\langle  Im(\tilde{T})\tilde{x},\tilde{x}  \rangle=0$, and  so using Lemma \ref{lem9} and Lemma \ref{lem5}, we get that
\begin{eqnarray*}
|\langle Tx,x \rangle| &\leq& \langle Re(\tilde{T}) \tilde{x},\tilde{x}  \rangle	\\
\Rightarrow 	|\langle Tx,x \rangle| &\leq& \langle |Re(\tilde{T}) |\tilde{x},\tilde{x}  \rangle	\\
\Rightarrow 	 |\langle Tx,x \rangle|^{2p} &\leq&  \langle |Re(\tilde{T}) |\tilde{x},\tilde{x}  \rangle ^{2p}	\\
\Rightarrow 	 |\langle Tx,x \rangle|^{2p} &\leq&  \langle |Re(\tilde{T}) |^{2p}\tilde{x},\tilde{x}  \rangle 	\\
\Rightarrow 	 |\langle Tx,x \rangle|^{2p} &\leq&  \langle |R|^{2p}\tilde{x},\tilde{x}  \rangle.\\
\end{eqnarray*}
Now proceeding similarly as in the proof of Theorem \ref{est6} and using Lemma \ref{lem5}, we obtain 
\[ \|Tx\|^{2p}  \leq \langle |S|^{2} \tilde{x},\tilde{x}  \rangle ^p \leq \langle |S|^{2p} \tilde{x},\tilde{x}  \rangle.\]

\noindent By convexity of $t^p, p\geq 1$, it follows that
\begin{eqnarray*}
 \left ( \alpha |\langle Tx ,x \rangle |^2+(1-\alpha)\|Tx\|^2 \right )^p &\leq& \left (\alpha |\langle Tx, x\rangle |^{2p}  + (1-\alpha) \| Tx \|^{2p} \right ) \\
&\leq&  \left (\alpha \langle |R|^{2p} \tilde{x},\tilde{x}  \rangle + (1-\alpha) \langle |S|^{2p} \tilde{x},\tilde{x}  \rangle \right ) \\
&=&  \langle \left (\alpha |R|^{2p}+(1-\alpha) |S|^{2p}   \right)  \tilde{x},\tilde{x} \rangle \\
&\leq& \left \| \alpha |R|^{2p}+(1-\alpha) |S|^{2p}  \right \|.
\end{eqnarray*}
Therefore, taking supremum over all unit vectors in $\oplus_{i=1}^n \mathcal{H}$, we get the desired inequality.
\end{proof}

We simply state the following result and omit its proof, as it can be completed using similar arguments as given in the proof of Theorem \ref{est8}.

\begin{theorem}\label{est9}
Let $T=(T_{ij})$ be an $n\times n$ operator matrix, where $T_{ij}\in \mathcal{B}(\mathcal{H})$. Let $f$ and $g$ be two non-negative continuous functions on $[0,\infty)$ such that $f(t)g(t)=t,$ $\forall ~ t \geq 0 $. Then
\begin{eqnarray*}
\|T\|_{\alpha} & \leq & \sqrt{ \left \|\alpha |R|^{2}+(1-\alpha) |S|^{2} \right \| },
\end{eqnarray*}
where $ r_{ij}=\begin{cases}
		\frac{1}{2}\left \| f^2(|T_{ii}|)  + g^2(|T_{ii}^*|) \right \| &\text{if} ~i=j \\
		\frac{1}{2} \left (\|f^2(|T_{ij}|)\|^{\frac{1}{2}} \|g^2(|T_{ij}^*|)\|^{\frac{1}{2}}+\|f^2(|T_{ji}|)\|^{\frac{1}{2}} \|g^2(|T_{ji}^*|)\|^{\frac{1}{2}} \right ) &\text{if} ~i\neq j,  \\
		\end{cases}$\\ $R=(r_{ij})_{n \times n}$ and  $S=(s_{ij})_{n \times n}$, $s_{ij} =\|T_{ij}\|$.
\end{theorem}

The following numerical radius inequality follows easily from Theorem \ref{est9} by using Lemma \ref{norm}.

\begin{cor}\label{est9_c1}
Let $T=(T_{ij})$ be an $n\times n$ operator matrix, where $T_{ij}\in \mathcal{B}(\mathcal{H})$. Let $f$ and $g$ be two non-negative continuous functions on $[0,\infty)$ such that $f(t)g(t)=t,$ $\forall ~ t \geq 0 $. Then
\begin{eqnarray*}
w(T) & \leq & \min_{0\leq \alpha \leq 1} \sqrt {\left \|\alpha |R|^{2}+(1-\alpha) |S|^{2} \right \| },
\end{eqnarray*}
where $R,S$ are same as described in Theorem \ref{est9}.
\end{cor}

\begin{remark}
In particular, if we consider $\alpha =1$ in Corollary \ref{est9_c1} then using Lemma \ref{lem}, we get 
\begin{eqnarray*}
w(T) & \leq & \min_{0\leq \alpha \leq 1} \sqrt {\left \|\alpha |R|^{2}+(1-\alpha) |S|^{2} \right \|}\\
&\leq & w(\tilde{T}),
\end{eqnarray*}
where $\tilde{T} =(\tilde{t_{ij}})_{n\times n}$, $\tilde{t_{ij}}=\begin{cases}
		\frac{1}{2}\left \| f^2(|T_{ii}|)  + g^2(|T_{ii}^*|) \right \| &\text{if} ~i=j \\
		\|f^2(|T_{ij}|)\|^{\frac{1}{2}} \|g^2(|T_{ij}^*|)\|^{\frac{1}{2}} &\text{if} ~i\neq j.
		\end{cases}$ \\
Note that the existing inequality in \cite[Th. 3.3]{BP} follows from Corollary \ref{est9_c1}.  
\end{remark}

\bibliographystyle{amsplain}

\end{document}